\documentclass[reqno]{amsart}
\usepackage{amssymb}  
\usepackage{amsmath} 
\usepackage{amsthm}  
\usepackage{bm} 
\usepackage{stmaryrd} 
\usepackage{eucal}     

\advance\hoffset by -1in
\advance\voffset by -1in

\newtheorem{thm}{Theorem}
\newtheorem{lem}[thm]{Lemma}

\newtheorem{conj}{Conjecture}

\newcommand{\be}{\begin{equation}}
\newcommand{\ee}{\end{equation}}

\newcommand{\la} {\langle}
\newcommand{\ra} {\rangle}
\newcommand{\se} {\subseteq}

\newcommand{\vf}{\varphi}

\newcommand{\s}{\sigma}

\renewcommand{\r}{\rho}
\renewcommand{\le}{\leqslant}
\renewcommand{\ge}{\geqslant}

\newcommand{\I}{\bm{1}}

\newcommand{\x}{\chi}
\newcommand{\rd}{\rightthreetimes}
\newcommand{\ov}{\overline}
\newcommand{\FF} {\mathbb{F}}

\newcommand{\ZZ} {\mathbb{Z}}

\newcommand{\RR} {\mathbb{R}}

\newcommand{\ba}[1] {\begin{array}{#1}}
\newcommand{\ea} {\end{array}}

\newcommand{\GL}{\mathop\mathrm{GL}\nolimits}
\newcommand{\PSL}{\mathop\mathrm{PSL}\nolimits}
\newcommand{\Sym}{\mathop\mathrm{Sym}\nolimits}

\newcommand{\SL}{\mathop\mathrm{SL}\nolimits}
\newcommand{\PsAut}{\mathop\mathrm{PsAut}\nolimits}

\newcommand{\OO}{\mathbb{O}}

\newcommand{\M}{\mathcal{M}}

\newcommand{\llb} {\llbracket}
\newcommand{\rrb} {\rrbracket}

\title{Multiplication formulas in Moufang loops}
\author{Alexander N. Grishkov and Andrei V. Zavarnitsine}

\thanks{Supported by FAPESP (proc. 2014/13730-0)}
\date{}

\begin{document}
\begin{abstract} Using groups with triality we obtain some general multiplication formulas in Moufang
loops, construct Moufang extensions of abelian groups, and describe the structure of
minimal extensions for finite simple Moufang loops over abelian groups.

{\sc Keywords:} Moufang loops, groups with triality, multiplication formula

{\sc MSC2010:} 08A05, 20E34, 20N05

\end{abstract}
\maketitle

\section{Introduction}

A {\em multiplication formula} for a loop $L$ with respect to its decomposition $L=MN$,
where $M,N$ are subloops of $L$,
is a rule of expressing the product of two elements of the form $mn$, with $m\in M$ and $n\in N$,
again in the same form that often assumes the multiplication within both $M$ and $N$
and some other information to be known.
Multiplication formulas generalize the concept of semidirect product
of groups and have been used to give explicit constructions of new loops.
Some examples of such formulas for Moufang loops can be found in \cite{che,raj,gag,gz_abc}.

Given a Moufang loop $L$ with a normal subloop $N$, we use its corresponding group
with triality $G$ to derive a general multiplication formula for the decomposition $L=MN$
where $M$ is a subloop. The formula takes especially simple form when the subgroup of $G$
corresponding to $N$ is abelian. We then give two important special examples of this
formula. The first one comes from modules for wreathlike triality groups
$G\times G\times G$. We obtain a criterion when these modules admit triality and write explicitly
the corresponding multiplication rule.
The second example may be viewed as a generalization of group action on associative algebras
to Moufang loop `action' on alternative algebras. Whenever a Moufang loop $M$ is mapped to
invertible elements $A^\times$ of an alternative algebra $A$, one can always construct
a semidirect product $M\rd U$ which is a Moufang loop, where $U$ is any subgroup of the additive group of $A$ invariant under the operators $L_{m,n}$ and $T_m$ for $m,n\in M$.

The Moufang loops that are upward extensions of abelian group by simple loops are of
special interest. We describe the structure of all known minimal such extensions for finite simple
noncyclic Moufang loops and conjecture that there are no others.

\section{Preliminaries}

For elements $x,y$ in a group, we set $[x,y]=x^{-1}y^{-1}xy$ and
define the conjugation operator $J_y : x\mapsto y^{-1}xy$. The symmetric group of
permutations of a set $X$ is denoted by $\Sym(X)$. The notation $V^{\oplus n}$ means the direct sum
of $n$ copies of a module $V$.

A loop $M$ is called a {\it Moufang loop} if $xy\cdot zx = (x\cdot yz)x$ for all $x,y,z \in M$,
where we use the common shorthand notation $x\cdot yz = x(yz)$.
For basic properties of Moufang loops, see \cite{bru}. Associative subloops of $M$ will be called subgroups.

If $(M,.)$ is a Moufang loop and $x,y\in M$ then we define
\be\label{ops} \ba{c}
 yL_x=x.y, \qquad yR_x=y.x, \qquad
 T_x=L_x^{-1}R_x,\qquad  P_x=L_x^{-1}R_x^{-1},  \\
L_{x,y}=L_xL_yL_{yx}^{-1}, \qquad R_{x,y}=R_xR_yR_{xy}^{-1}.
\ea \ee
We use the notation $\llb x,y\rrb=x^{-1}.y^{-1}.x.y $ instead of $[x,y]$ to
denote the commutator in $M$. We also set $(x,y,z)=(x.yz)^{-1}(xy.z)$.

A {\it pseudoautomorphism} of a Moufang loop
$(M,.)$  is a bijection $A: M\to M$
with the property that there exists an element $a
\in M$ such that
$$xA.(yA.a)=(x.y)A.a \qquad \text{for all} \quad x,y \in M. $$
This element $a$ is called a {\it companion} of $A$.
The set of pairs
$(A,a)$, where $A$ is a pseudoautomorphism of $Q$ with companion $a$,
is a group with respect to the operation
$$(A,a)(B,b)=(AB,aB.b).$$
This group is denoted by $\PsAut(M)$.
It is known \cite[Lemma VII.2.2]{bru} that $(T_x,x^{-3})$ and
$(R_{x,y},\llb x,y\rrb)$ belong to $\PsAut(M)$.

A group $G$ possessing automorphisms $\rho$ and $\sigma$ that satisfy $\rho^3=\sigma^2=(\rho\sigma)^2=1$ is
called a {\it group with triality $\langle\rho,\sigma\rangle$} if
$$
(x^{-1}x^{\sigma})(x^{-1}x^{\sigma})^\rho(x^{-1}x^{\sigma})^{\rho^2}=1
$$
for every $x$ in $G$. In a group $G$ with triality $S=\langle\rho,\sigma\rangle$,
the set
$$\M(G)=\{ x^{-1}x^{\sigma}\ |\ x\in G\}$$
is a Moufang loop with respect to the multiplication
\begin{equation} \label{loop_mult}
m.n=m^{-\rho} n  m^{-\rho^2}
\end{equation}
for all $n,m\in \M(G)$. Conversely, every Moufang loops arises so from a suitable group with triality.
For more information on the relation between groups with triality and Moufang loops, see \cite{gz_tri,dor}.

$S$-invariant subgroups of a group $G$ with triality (for short, $S$-subgroups) are also groups triality $S$,
whose corresponding Moufang loops are subloops of $\M(G)$.
If $G$ is a group with triality then, for $g\in G$, we define

$$
\vf(g)=g^{-\r}g^{\r^2}.
$$

\begin{lem}\label{gzt} Let $G$ be a group with triality
$S=\langle\rho,\sigma\rangle$.  Denote $M=\M(G)$ and $H=C_G(\s)$.
Then, for all $m,n,l\in M$ and $h\in H$, we have
\begin{enumerate}
\item[$(i)$] $m,m^\r,m^{\r^2}$ pairwise commute;
\item[$(ii)$] $m^{-\r} nm^{-\r^2}=n^{-\r^2} m n^{-\r}$;
\item[$(iii)$] $[n^{\r^2},m^{-\r}]=[n^{-\r},m^{\r^2}]\in H$;
\item[$(iv)$] $m.n.m = mnm$;
\item[$(v)$] $\vf(H)\se M$ and $\vf(M)\se H$;
\item[$(vi)$] $M^h=M$; moreover, $(J_h,\vf(h))\in \PsAut(M)$;
\item[$(vii)$] for $k=[m^\r,n^{-\r^2}]$, we have  $(J_k,\vf(k))=(R_{m,n},\llb m,n\rrb)$;
\item[$(viii)$] for $k=\vf(m)$, we have $(J_k,\vf(k))=(T_m,m^{-3})$;
\item[$(ix)$] $(l,m,n)=[k,l]^{l^{\r^2}}$, where $k=[n^{-\r},m^{\r^2}]$.
\end{enumerate}
\end{lem}
\begin{proof} $(i)$--$(viii)$ See \cite[Lemmas 2,4]{gz_tri}.

$(ix)$ We have
\begin{multline*}
(l,m,n)=(l.mn)^{-1}(lm.n)=(mn)^{-1}(m.nl^{-1}).l=l^{-1}L_{n,m}.l\\
=l^{-1}J_k.l=l^{-\r^2}(l^{-1}J_k)ll^{\r^2}=[k,l]^{l^{\r^2}},
\end{multline*}
where the second equality follows from multiplying both sides by $l^{-1}$ and using the Moufang identities. We have also used $(vii)$ keeping in mind that $L_{n,m}=R_{n^{-1},m^{-1}}$, see Lemma \ref{dxy}$(ii)$.
\end{proof}

\section{Auxiliary results}

\begin{lem}\label{act_x}
Let $G$ be a group with triality $S=\langle\rho,\sigma\rangle$. Denote $M=\M(G)$ and $H=C_G(\s)$. We have
\begin{enumerate}
\item[$(i)$] $g^{-1}mg^\s \in M$ for every $g\in G$ and every $m\in M$.
\item[$(ii)$] The map $\x:G\to \Sym(M)$, $g\mapsto \x_g$, given by $m\x_g = g^{-1}mg^\s$, $m\in M$, is a group homomorphism with $\ker\x = C_H(M)$.
\end{enumerate}
\end{lem}
\begin{proof} $(i)$ Since $m\in M$, we have $m=x^{-1}x^\s$ for some $x\in G$. Then $g^{-1}mg^\s=(xg)^{-1}(xg)^\s\in M$.

$(ii)$ Clearly, $\x_g\in \Sym(M)$ and $\x_{g_1}\x_{g_2}=\x_{g_1g_2}$. If $g\in \ker\x$ then $1=1\x_g=g^{-1} g^\s$, which yields $g\in H$. Moreover, $m = m\x_g=g^{-1}mg$ for very $m\in M$, and so $g\in C_H(M)$. The reverse inclusion is obvious.
\end{proof}

The permutation action $\x$ of $G$ on $M$ defined in Lemma \ref{act_x}$(ii)$ has the following implications.
First, the three subsets $M,M^\r,M^{\r^2}$ of $G$ are naturally permuted by $S$ and
their corresponding action $\x$ on $M$ is given by the following.
\begin{lem}\label{m_x} For every $m\in M$, we have
$$\x_m=P_m; \qquad \x_{m^\r}=L_m; \qquad \x_{m^{\r^2}}=R_m.$$
\end{lem}
\begin{proof} For every $n\in M$, we have
\begin{align*}
&n\x_m=m^{-1}nm^{-1} = m^{-1}.n.m^{-1}=nP_m;\\
&n\x_{m^\r}=m^{-\r}nm^{-\r^2} = m.n = nL_m;\\
&n\x_{m^{\r^2}}=m^{-\r^2}nm^{-\r} = n.m = nR_m\\
\end{align*}
by Lemma \ref{gzt}, and the claim follows.\end{proof}
In particular, Lemma \ref{m_x} demonstrates how the classical triality for Moufang loops
\be
\ba{c}
P_{x}\ \stackrel{\r}{\longmapsto} \ L_{x} \
\stackrel{\r}{\longmapsto} \ R_{x} \
\stackrel{\r}{\longmapsto} \ P_{x},\\
P_{x}\ \stackrel{\s}{\longmapsto} \ P_{x}^{-1}, \qquad
L_{x}\ \stackrel{\s}{\longmapsto}\ R_{x}^{-1}, \qquad
R_{x}\ \stackrel{\s}{\longmapsto}\ L_{x}^{-1}.
\ea \ee
can be observed from the permutation action $\x$ of $G$ on $M$.

Also, the action $\x$ restricted to the subgroups $H,H^\r,H^{\r^2}$, where $H=C_G(\s)$, is related to
the induced pseudoautomorphisms $J_h$, $h\in H$, of $M$ and their companions $\vf(h)$, see Lemma \ref{gzt}$(vi)$, in the following manner.

\begin{lem}\label{h_x} For every $h\in H$, we have
$$\x_h=J_h; \qquad \x_{h^\r}=J_hR_{\vf(h)}; \qquad \x_{h^{\r^2}}=J_hL_{\vf(h)^{-1}}.$$
\end{lem}
\begin{proof} For every $m\in M$, we have
\begin{align*}
&m\x_h=h^{-1}mh^\s = h^{-1}mh = mJ_h;\\
&m\x_{h^\r}=h^{-\r}m h^{\r^2} = (h^{-\r}h^{\r^2})^{-\r^2}h^{-1}mh(h^{-\r}h^{\r^2})^{-\r} = mJ_h.\vf(h);\\
&m\x_{h^{\r^2}}=h^{-\r^2}mh^\r = (h^{-\r}h^{\r^2})^\r h^{-1}mh(h^{-\r}h^{\r^2})^{\r^2} = \vf(h)^{-1}.mJ_h\\
\end{align*}
by Lemma \ref{gzt}, and the claim follows. \end{proof}

\begin{lem}\label{dxy} Let $M$ be a Moufang loop and let $x,y,m\in M$. Then
\begin{enumerate}
\item[$(i)$] $m^{-1}(mx\cdot y)=xm^{-1}\cdot my=(x\cdot ym^{-1})\cdot m$;
\item[$(ii)$] $R_{x,y}=L_{x^{-1},y^{-1}}=R_xL_y^{-1}R_x^{-1}L_y$;
\item[$(iii)$] $x^{-1}\cdot(xy^{-1}\cdot m)y=y^{-1}x^{-1}\cdot(xm \cdot y)=(y^{-1}\cdot mx^{-1})\cdot xy = y^{-1}(m\cdot yx^{-1})\cdot x$;
\item[$(iv)$] The operator $D_{x,y}:m\mapsto x^{-1}\cdot(xy^{-1}\cdot m)y$ is a pseudoautomorphism of M with companion $y^{-1}xy^{-2}x^{-1}$;
\item[$(v)$] In $\PsAut(M)$, we have
$$(D_{x,y},y^{-1}xy^{-2}x^{-1})=(L_{x,y^{-1}},\llb x^{-1},y\rrb)\,(T_y,y^{-3})\,(L_{y,x},\llb y^{-1},x^{-1}\rrb).$$
\end{enumerate}
\end{lem}
\begin{proof} $(i)$ The first equality holds, since the left Moufang identity implies
$$
m\cdot(xm^{-1}\cdot my)=mxm^{-1}m\cdot y=mx\cdot y.
$$ The second equality is dual to the first
one (i.e. follows by inversion).

$(ii)$ The first equality is known \cite[Lemma VII.5.4]{bru}. Applying $(i)$, we have
$$
mR_{x,y}=(mx\cdot y)\cdot y^{-1}x^{-1}=y\cdot(y^{-1}\cdot mx)x^{-1}=mR_xL_{y^{-1}}R_{x^{-1}}L_y
$$
$(iii)$ In order to show the first equality, it suffices to prove that
$xy^{-1}\cdot m = x(y^{-1}x^{-1}\cdot(xm\cdot y))\cdot y^{-1}$. Using the Moufang identities, the right-hand side can be rewritten as
\begin{align*}
&x(y^{-1}x^{-1}\cdot (x\cdot x^{-1}(xm\cdot y)))\cdot y^{-1}=(xy^{-1}\cdot x^{-1}(xm\cdot y))\cdot y^{-1}\\
=&\, x(y^{-1}x^{-1}\cdot (xm\cdot y) y^{-1})=xy^{-1}x^{-1}x\cdot m = xy^{-1}\cdot m.
\end{align*}
The third equality is dual to the first one. In order to show the second one, it suffices to prove that
$xm\cdot y = xy\cdot (y^{-1}\cdot mx^{-1})\cdot xy$. The right-hand side equals
$x(mx^{-1}\cdot xy) = xmx^{-1}x\cdot y = xm\cdot y$.

$(iv)$ follows from $(v)$.

$(v)$ We first show that $D_{x,y}=L_{x,y^{-1}}T_yL_{y,x}$. By the second expression in $(iii)$,
we have $D_{x,y}=L_xR_yL_{xy}^{-1}$. On the other hand, $(ii)$ implies
$$
L_{x,y^{-1}}T_yL_{y,x}=(R_x^{-1}L_y^{-1}R_xL_y)L_{y^{-1}}R_y(L_yL_xL_{xy}^{-1})
=R_x^{-1}L_y^{-1}R_xR_yL_yL_xL_{xy}^{-1}.
$$
Hence, it remains to show that $R_x^{-1}L_y^{-1}R_xR_yL_yL_x=L_xR_y$. This follows from
$R_xR_yL_yL_x=L_yR_xL_xR_y$, which in turn follows by
$$
x(y\cdot mx\cdot y) = x(ym\cdot xy) = (x\cdot ym \cdot x)y
$$
for every $m$ due the Moufang identities.

For the companion, we have

$$
(\llb x^{-1},y \rrb T_y\cdot y^{-3})L_{y,x}\cdot \llb y^{-1},x^{-1}\rrb=
(y^{-1}\cdot xy^{-1}x^{-1}y\cdot y)y^{-3}\cdot yxy^{-1}x^{-1}=y^{-1}xy^{-2}x^{-1},
$$
where we have used the fact that $L_{x,y}$ acts as the identity on the
subgroup $\langle x,y \rangle$ due to diassociativity.

\end{proof}

\section{A multiplication formula}

\begin{lem}\label{gen_m} Let $G$ be a group with triality $S$ and let $m,n,u,w\in \M(G)$. Then
\be\label{munw}
(m.u).(n.w)=(m.n).x,
\ee
where
\be\label{xx}
x=u^{-\rho n ^{-\rho}m^{\rho^2}}w^{[n^{\rho^2},m^{-\rho}]}u^{-\rho^2n^{\rho^2}m^{-\rho}}\in \M(G);
\ee
and
\be\label{inv}
(m.u)^{-1}=m^{-1}.y, \quad\text{where}\quad y=u^{\r m^{-1}}u^{\r^2m}.
\ee
\end{lem}
\begin{proof}
The left-hand side of (\ref{munw}) can be expanded using (\ref{loop_mult}) and Lemma \ref{gzt}$(iii)$ as follows
\begin{align*}
&m^{-\r}um^{-\r^2}.n^{-\r}wn^{-\r^2}=(mu^{-\r}m^{\r^2})n^{-\r}wn^{-\r^2}(m^\r u^{-\r^2}m)=(mn^{-\r}m^{\r^2})\, \times\\
\times&\,m^{-\r^2}n^\r u^{-\r}n^{-\r}m^{\r^2}[m^{\r^2},n^{-\r}]w[n^{\r^2},m^{-\r}]m^\r n^{-\r^2}
u^{-\r^2}n^{\r^2}m^{-\r}(m^\r n^{-\r^2}m)\\
=&(m^{-\r}nm^{-\r^2})^{-\r}u^{-\r n^\r m^{\r^2}}w^{[n^{\r^2},m^{-\r}]}u^{-\r^2 n^{\r^2}m^{\r}}(m^{-\r}nm^{-\r^2})^{-\r^2}=(m.n).x
\end{align*}
Observe that $x=g^{-1}w^hg^\sigma$, where $h=[n^{\rho^2},m^{-\rho}]\in C_G(\s)$ and $g=u^{\r n^{-\r}m^{\r^2}}$. Hence, $x\in \M(G)$ by Lemma \ref{act_x}.

The inversion formula (\ref{inv}) follows from (\ref{xx}) by setting $n=m^{-1}$ and using Lemma \ref{gzt}$(i)$.
\end{proof}

This result can be applied in the following situation.

\begin{lem}\label{m_f} Let $K,V$ be $S$-subgroups of a group $G$ with triality.
Suppose that $V$ is normal in $G$ and $G=KV$. Then (\ref{munw}) is a multiplication formula
in the corresponding Moufang loop $M=\M(G)$ with respect to its decomposition $M=\M(K)\M(V)$,
where $\M(V)$ is a normal subloop of $M$ and $m,n\in \M(K)$, $u,w\in \M(V)$.
\end{lem}
\begin{proof} We only need to show that the element $x$ from (\ref{xx}) lies in $\M(V)$. This
follows from applying the natural homomorphism $M\to M/\M(V)$ to both sides of (\ref{munw}).
\end{proof}

Lemma \ref{m_f} can be used to obtain a multiplication formula for any Moufang loop $L$
with respect to any decomposition $L=MN$, where $M,N$ are subloops of $L$ with $N$ normal, because
there always exists a group with triality $G$ as stated in the lemma such that $\M(G)=L$,
$\M(K)=M$ and $\M(N)=V$. In fact, any group $G$ with triality such that $[G,S]=G$ and
$\M(G)=L$ will have this property, see \cite{dor}.

The case where $V$ is abelian deserves special attention.

\begin{lem}\label{m_f_ab} If, under the assumptions of Lemma \ref{m_f},
$V$ is an abelian subgroup of $G$ then
the multiplication and inversion formulas in $\M(G)$ take the form
\be\label{xx_ab}
\ba{r@{\,}l}
(m.u).(n.w)&=(m.n).x, \qquad x=uD_{m,n}+wL_{n,m}; \\
(m.u)^{-1}&=m^{-1}.y, \qquad y=-uT_{m}^{-1}
\ea\ee
for all $m,n\in \M(K)$, $u,w\in \M(V)$.
\end{lem}
\begin{proof} First, observe that $J_h=L_{n,m}$ for $h=[n^{\r^2},m^{-\r}]$ by Lemmas
\ref{gzt}$(iii,vii)$ and \ref{dxy}$(ii)$. In particular, $w^h=wL_{n,m}$.
Also, if we set $w=1$ in (\ref{munw}),
we have $x=(m.n)^{-1}.((m.u).n)=uD_{m,n}$. Hence, with $w=1$ the expression (\ref{xx}) gives
\be\label{x_dmn}
uD_{m,n}=u^{-\rho n ^{-\rho}m^{\rho^2}}u^{-\rho^2n^{\rho^2}m^{-\rho}}.
\ee
Since $(m.u)^{-1}=u^{-1}.m^{-1}=m^{-1}.u^{-1}T_m^{-1}$, we have the expression for $y$.

Assume now that $V$ is abelian. Since $V$ is a normal $S$-subgroup, all three factors on the right-hand
side of (\ref{xx}) belong to $V$. Permuting the last two factors and using (\ref{x_dmn}),
we obtain (\ref{xx_ab}) in the additive notation.
\end{proof}

\begin{lem}\label{ass} Suppose in a Moufang loop $L=MU$ with a subloop $M$ and a normal abelian subgroup $U$, the multiplication is given by (\ref{xx_ab}) for all $m,n\in M$, $u,w\in U$.
Then $(l,u,w)=1$ for all $l\in L$, $u,w\in U$.
\end{lem}
\begin{proof} We use multiplicative notation in $U$. Using (\ref{xx_ab}),
we see that $mu\cdot w = m\cdot uw$ for every $m\in M$. Hence, if $l=mv$ for $m\in M$, $v\in U$, we have
$lu\cdot w=(m\cdot vu)w=m(vu\cdot w)$ and $l\cdot uw=m(v\cdot uw)$. The claim follows, since $U$ is associative.
\end{proof}

Lemma \ref{ass} shows that not every Moufang loop $L$ with decomposition $L=MU$, where $U$ is normal abelian, has multiplication (\ref{xx_ab}). For example, the abelian-by-cyclic Moufang loops constructed in \cite{gz_abc} have this decomposition, but do not satisfy the conclusion of the lemma. The reason is that for those loops, in the corresponding group with triality, the normal $S$-subgroup $V$ such that $\M(V)=U$ is not abelian.  The following generalization of Lemma \ref{ass} holds.

\begin{lem} \label{ab_nor} Let $G$ be a group with triality $S$ and let $V$ be an abelian $S$-subgroup of $G$. Denote $U=\M(V)$, which is an abelian subgroup of $L=\M(G)$. Then we have $(L,U,U)=1$.
\end{lem}
\begin{proof} Let $l\in L$ and $u,v\in U$. Recall that $U\se V$. By Lemma \ref{gzt}$(ix)$, we have $(l,u,v)=[k,l]^{l^{\r^2}}$, where $k=[v^{-\r},u^{\r^2}]$. Since $V$ is $S$-invariant, we have $v^{-\r},u^{\r^2}\in V$. Since $V$ is abelian, we have $k=1$. Therefore, $(l,u,v)=1$ and the claim follows.
\end{proof}

In the next sections, we will show how the action of operators $D_{m,n}$, $L_{n,m}$, and $T_m$
that appear in the multiplication formula (\ref{xx_ab})
can be restored from the group action on modules or
linked with the inner multiplication of an alternative algebra on its additive group.

\section{Modules for wreathlike triality groups \label{mtg}}

Let $G$ be a group. It it known that $T=G\times G\times G$ is a group with triality $S=\la\r,\s\ra$
(which we call {\em wreathlike} following \cite{dor}) with respect to the natural action
$(g_1,g_2,g_3)^\r=(g_3,g_1,g_2)$, $(g_1,g_2,g_3)^\s=(g_2,g_1,g_3)$, whose corresponding Moufang loop is
isomorphic to $G$.

Let $R$ be a commutative associative unital ring and let $V$ be an $RG$-module which is free of rank $n$
as an $R$-module. The outer tensor product $W=V\#\, V\#\, V$ is an $RT$-module
\cite[Definition VII.43.1]{cr} which admits the natural action of $S$ by
$(v_1 \#\, v_2 \#\, v_3)^\r=v_3 \#\, v_1 \#\, v_2$,
$(v_1 \#\, v_2 \#\, v_3)^\s=v_2 \#\, v_1 \#\, v_3$
that is compatible with the action of $T$ on $W$ in the sense that $(w^t)^\tau=(w^\tau)^{(t^\tau)}$ for all
$w\in W$, $t\in T$, $\tau \in S$, and hence is extended to $A=T\rd W$. Our aim is to determine when $A$ is
a group with triality $S$.

\begin{lem}\label{dim} The group $A$ constructed above has triality $S$ if and only if $n\le 2$.
If $n=1$ then $\M(A)\cong G$.
\end{lem}
\begin{proof} Let $e_1,\ldots,e_n$ be an $R$-basis of $V$. For $t=(g_1,g_2,g_3)\in T$ we have
$t^{-1}t^\s=(g_1^{-1}g_2,g_2^{-1}g_1,1)$. Hence, $\M(T)=\{(g^{-1},g,1)\mid g\in G\}$.
By \cite[Lemma 4]{gz_abc}, $A$ has triality $S$ if and only if $W$ has triality $S_{(m)}$
for all $m\in \M(T)$, where $S_{(m)}=\la\r^2m\r^2,\s\ra$. Let $m=(g^{-1},g,1)$. Then
$\r^2m\r^2=\r(g,1,g^{-1})$ and $\r m^{-1}\r=\r^2(1,g,g^{-1})$. Hence, $A$ has triality if and only if,
for every basis element $e_{ijk}=e_i\#\, e_j \#\, e_k$ of $W$ and every $g\in G$,
$$
e_{ijk}^{(1-\s)(1+\r(g,1,g^{-1})+\r^2(1,g,g^{-1}))}=0.
$$
Expanding and acting on both sides by $(1,1,g)$, we have
$$
(e_{ijk} - e_{jik} )^{(1,1,g)}+(e_{kij} - e_{kji} )^{(g,1,1)}+(e_{jki} - e_{ikj} )^{(1,g,1)}=0.
$$
Let $e_kg=\sum_s g_{ks}e_s$, where $g_{ks}\in R$. Then the condition is rewritten as
\be\label{cond_tri}
\sum_{s=1}^n g_{ks}(e_{ijs}-e_{jis}+e_{sij}-e_{sji}+e_{jsi}-e_{isj})=0
\ee
which must hold for all $g\in G$, $1\le i,j,k\le n$.  In particular, setting $g=1$ gives
$$
e_{ijk}-e_{jik}+e_{kij}-e_{kji}+e_{jki}-e_{ikj}=0.
$$
This does not hold, say, for $(i,j,k)=(1,2,3)$, if $n\ge 3$, since the basis elements are linearly
independent. However, if $n\le 2$ then (\ref{cond_tri}) is satisfied, since at least two of $i,j,s$ will coincide.

In the case $n=1$, the normal subgroup $W$ of $A$ is spanned by $e_{111}$ on which $S$ acts trivially.
Hence $\M(A)\cong \M(A/W) \cong \M(T) \cong G$.
\end{proof}

The case $n=2$ is of main interest and we consider it now in more detail.

\begin{thm}\label{gd} Let $G$ be a subgroup of $\GL_2(R)$.
Let $V$ be the free $R$-module of rank $2$ with the natural
action ``$\,\circ$'' of $G$. Denote by $G\rd V$ the set of pairs $(g,u)$ for $g\in G$, $u\in V$. Then
with respect to the operation
\be\label{m_g}
(g,u)\cdot(h,w)=(gh,u\circ (\det h) gh^{-2}g^{-1} + w\circ [h^{-1},g^{-1}])
\ee
$G\rd V$ becomes a Moufang loop with identity $(1,0)$ and inversion
$$
(g,u)^{-1}=(g^{-1},-u\circ (\det g)^{-1}g^2)
$$
Moreover, this Moufang loop is isomorphic to $\M(A)$, where $A=T\rd W$ is the triality group constructed
above with respect to the $RG$-module $V$.
\end{thm}
\begin{proof} Let $e_1,e_2$ be an $R$-basis of $V$. Then $e_{ijk}=e_i\#\,e_j\#\,e_k$,
$i,j,k=1,2$, is a basis of $W=V\#\,V\#\,V$. As above, let $T=G\times G\times G$ with $A=T\rd W$
admitting the natural action of $S$. By Lemma \ref{dim}, $A$ is a group with triality. We see
that $\M(W)$ is spanned by $f_1=e_{121}-e_{211}$ and $f_2=e_{122}-e_{212}$. Let $m,n\in \M(T)$ and $u,w\in \M(W)$.
Lemma \ref{gen_m} implies that in the loop $\M(A)$ we have $(m.u).(n.w)=(m.n).x$, where
$$
x=u^{-\rho n ^{-\rho}m^{\rho^2}-\rho^2n^{\rho^2}m^{-\rho}} + w^{[n^{\rho^2},m^{-\rho}]}.
$$
Let $m=(g^{-1},g,1)$ and $n=(h^{-1},h,1)$ for suitable $g,h\in G$. Then $n^{\r^2}=(h,1,h^{-1})$ and
$m^{-\r}=(1,g,g^{-1})$. Therefore,
$$
[n^{\rho^2},m^{-\rho}] = (h^{-1},1,h)(1,g^{-1},g)(h,1,h^{-1})(1,g,g^{-1})=(1,1,[h^{-1},g^{-1}]).
$$
Observe that $f_k=(e_1\#\,e_2-e_2\#\,e_1)\#\,e_k$, $k=1,2$.
Hence, the matrix of action of $(1,1,[h^{-1},g^{-1}])$ in $\{f_1,f_2\}$ coincides with
the matrix of action of $[h^{-1},g^{-1}]$ in $\{e_1,e_2\}$.

Similarly, we have
\begin{multline}\label{lin}
-\rho n ^{-\rho}m^{\rho^2}-\rho^2n^{\rho^2}m^{-\rho}=-n^{-1}m^\r\r-nm^{-\r^2}\r^2
=-(h,h^{-1},1)(1,g^{-1},g)\r\\
-(h^{-1},h,1)(g^{-1},1,g)\r^2=-(h,h^{-1}g^{-1},g)\r-(h^{-1}g^{-1},h,g)\r^2,
\end{multline}
and this operator send the basis element $e_{ijk}$ to
$$
-e_kg\#\,e_ih\#\,e_jh^{-1}g^{-1}-e_jh\#\,e_kg\#\,e_ih^{-1}g^{-1}.
$$
Therefore, $f_k=(e_1\#\,e_2-e_2\#\,e_1)\#\,e_k$ is sent to

\be\label{st}
(e_kg\#\,e_2h-e_2h\#\,e_kg)\#\,e_1h^{-1}g^{-1}-
(e_kg\#\,e_1h-e_1h\#\,e_kg)\#\,e_2h^{-1}g^{-1}.
\ee
Observe that, for any $a,b\in \mathrm{M}_2(R)$ with $a=(a_{jl})$, $b=(b_{jl})$ and any $k,i=1,2$, we have

\setlength\multlinegap{.5\multlinegap}
\begin{multline*}
e_ka\#\,e_ib-e_ib\#\,e_ka=(a_{k1}e_1+a_{k2}e_2)\#\,(b_{i1}e_1+b_{i2}e_2)-\\
(b_{i1}e_1+b_{i2}e_2)\#\,(a_{k1}e_1+a_{k2}e_2)=(a_{k1}b_{i2}-a_{k2}b_{i1})(e_1\#\,e_2-e_2\#\,e_1).
\end{multline*}
Denoting $C^{a,b}_{ki}=a_{k1}b_{i2}-a_{k2}b_{i1}$, we check that
\be\label{m_c}
\left(
  \begin{array}{cc}
    C^{a,b}_{12} & -C^{a,b}_{11} \\
    C^{a,b}_{22} & -C^{a,b}_{21} \\
  \end{array}
\right)=ab^*,
\ee
where $b^*$ is the adjoint matrix of $b$, i.e. $bb^*=(\det b)I$ with $I$ the identity matrix.
Using this notation, (\ref{st}) can be rewritten as
\begin{multline*}
C^{g,h}_{k2}(e_1\#\,e_2-e_2\#\,e_1)\#\,(e_1h^{-1}g^{-1})-
C^{g,h}_{k1}(e_1\#\,e_2-e_2\#\,e_1)\#\,(e_2h^{-1}g^{-1})\\
=(C^{g,h}_{k2}(h^{-1}g^{-1})_{11}-C^{g,h}_{k1}(h^{-1}g^{-1})_{21})f_1 +
 (C^{g,h}_{k2}(h^{-1}g^{-1})_{12}-C^{g,h}_{k1}(h^{-1}g^{-1})_{21})f_2.
\end{multline*}
Therefore, (\ref{m_c}) implies that the matrix of the transformation (\ref{lin}) in the basis
$\{f_1,f_2\}$ equals
\begin{multline*}
\left(
  \begin{array}{cc}
    C^{g,h}_{12}(h^{-1}g^{-1})_{11}-C^{g,h}_{11}(h^{-1}g^{-1})_{21} & C^{g,h}_{12}(h^{-1}g^{-1})_{12}-C^{g,h}_{11}(h^{-1}g^{-1})_{21} \\
    C^{g,h}_{22}(h^{-1}g^{-1})_{11}-C^{g,h}_{21}(h^{-1}g^{-1})_{21} & C^{g,h}_{22}(h^{-1}g^{-1})_{12}-C^{g,h}_{21}(h^{-1}g^{-1})_{21} \\
  \end{array}
\right)\\
=gh^*h^{-1}g^{-1}=(\det h)gh^{-2}g^{-1}.
\end{multline*}
It is clear from this discussions that under the map $m.u\mapsto (g,u)$, where $m=(g^{-1},g,1)$,
the multiplication formula from Lemma \ref{m_f_ab} for the loop $\M(A)$ takes the required
form (\ref{m_g}) thus giving the claimed isomorphism. The assertions about the identity and inversion
readily follow.
\end{proof}

The loop $G\rd V$ from Theorem \ref{gd} has a naturally embedded subgroup isomorphic to $G$ with
elements of the form $(g,0)$ and a normal subgroup isomorphic to $V$ with elements $(1,v)$.
We will identify $G$ and $V$ with these corresponding subgroups and call $G\rd V$ their
(outer) {\em Moufang semidirect product}.

It is easily checked that $G\rd V$ is nonassociative if and only if $G$ is nonabelian.
Let $G_0$ be the scalar subgroup of $G$. Then $G_0$ is a central (hence, normal) subloop of $G\rd V$.
The factor loop $(G\rd V)/G_0$, which we denote by $\ov{G}\rd V$ with $\ov{G}=G/G_0$, has a `projective'
analog of the formulas (\ref{m_g}) with $g,h$ replaced by $\ov{g},\ov{h}$, because the operators
$(\det h)gh^{-2}g^{-1}$ and $[h^{-1},g^{-1}]$ are constant on the cosets $G:G_0$. We note that
$\ov{G}\rd V$ can be nonassociative even when $\ov{G}$ is abelian.

The existence of the following Moufang semidirect product of finite simple groups follows from
the above construction along with the subgroup structure \cite{bhr} of $\GL_2(q)$:

\begin{enumerate}
\item[$\bullet$] $\PSL_2(q)\rd (\FF_q\oplus \FF_q)$, where $q\ge 4$ is a prime power;
\item[$\bullet$] $A_5\rd (\FF_p\oplus \FF_p)$, where $p\equiv \pm 1\pmod{10}$ is prime;
\item[$\bullet$] $A_5\rd (\FF_{p^2} \oplus \FF_{p^2})$, where $p\equiv \pm 3\pmod{10}$ is prime.
\end{enumerate}

\section{Embedding in a Cayley algebra}

A loop of shape $\GL_2(R)\rd (R \oplus R)$ also appears as a parabolic subloop of the invertible elements of the split Cayley
$R$-algebra $\OO=\OO(R)$.  We recall that $\mathbb{O}$ can be defined
as the  set of all {\em Zorn matrices}
$$
\left( \begin{array}{cc}
  a & {\bf v} \\
 {\bf w} &  b \\
\end{array}
\right), \ \  a, b \in R, \ \ {\bf v},{\bf w}\in R^3
$$
with the natural structure of a free $R$-module and
multiplication given by the rule
\begin{equation} \label{cayley_mult}
\begin{array}{r@{}l}
\left( \begin{array}{cc}
  a_1 & {\bf v}_1 \\
 {\bf w}_1 &  b_1 \\
\end{array}
\right)\cdot
\left( \begin{array}{cc}
 a_2 & {\bf v}_2 \\
 {\bf w}_2 & b_2 \\
\end{array}
\right) =& \\[20pt]
\left( \begin{array}{cc}
 a_1 a_2+ {\bf v}_1 \cdot {\bf w}_2 &  a_1 {\bf v}_2 + b_2 {\bf v}_1 \\
 a_2 {\bf w}_1 +  b_1 {\bf w}_2 &  {\bf w}_1\cdot {\bf v}_2 +  b_1 b_2 \\
\end{array}
\right)& +
\left( \begin{array}{cc}
 0 & - {\bf w}_1\times {\bf w}_2 \\
 {\bf v}_1\times {\bf v}_2 & 0 \\
\end{array}
\right),
\end{array}
\end{equation}
where, for ${\bf v}=(v_1,v_2,v_3)$ and ${\bf w}=(w_1,w_2,w_3)$ in
$R^3$, we denoted
$$
\begin{array}{c}
{\bf v}\cdot{\bf w}= v_1w_1+v_2w_2+v_3w_3\in R,\\
{\bf v}\times{\bf w}=(v_2w_3-v_3w_2,v_3w_1-v_1w_3,v_1w_2-v_2w_1)\in R^3.
\end{array}
$$

It is well-known that $\mathbb{O}$ is an alternative algebra and the set of its invertible
elements $\mathbb{O}^\times$ is a Moufang loop. The parabolic subloop of $\mathbb{O}^\times$
can be identified with the set of elements of the form
\be\label{p_e}
\left( \begin{array}{cc}
  a_{11} & (0,a_{12},r_1) \\
 (r_2,a_{21},0) &  a_{22} \\
\end{array}
\right),
\ee
where $(a_{ij})\in GL_2(R)$ and $r_1,r_2\in R$.

\begin{lem} Let $V$ be the free $R$-module of rank $2$ with the natural action ``$\,\circ$'' of $\GL_2(R)$.
Then the outer Moufang semidirect product $\GL_2(R)\rd V$ is isomorphic to a parabolic subloop of the Cayley
algebra $\OO(R)$.
\end{lem}
\begin{proof} Identifying the element (\ref{p_e}) with the pair $(a,r\circ a^{-1})$, where $a=(a_{ij})$,
$r=(r_1,r_2)$, one easily checks using (\ref{cayley_mult}) that such pairs are multiplied in $\OO^\times$
as follows
$$
(a,r\circ a^{-1})\cdot(b,s\circ b^{-1}) = (ab,(r\circ b^*+s\circ a)\circ (ab)^{-1}),
$$
where $b^*$ is the adjoint matrix of $b$. Setting $u=r\circ a^{-1}$ and $w=s\circ b^{-1}$, we have
$$
(a,u)\cdot(b,w) = (ab,u\circ ab^*b^{-1}a^{-1}+w\circ bab^{-1}a^{-1}),
$$
which coincides with the formula (\ref{m_g}).
\end{proof}

As a corollary we see that the Moufang semidirect product $\PSL_2(q)\rd (\FF_q\oplus \FF_q)$ is isomorphic
to a maximal parabolic subloop of the finite simple Moufang loop $M(q)$, see \cite{gz_max}.

\section{The action of Moufang loops on abelian groups}

Let $R$ be a commutative associative unital ring and let $A$ be an alternative $R$-algebra with unit $\I$. It is well known \cite[Lemma 2.3.7]{she} that the Moufang
identities hold in $A$. In particular, the set of invertible elements
$A^\times$ of $A$ is a Moufang loop. For every $x\in A^\times$, the maps $L_x:a\mapsto xa$, $R_x:a\mapsto ax$ are invertible linear operators of~$A$ and we may also define the following operators
$$
T_x=L_x^{-1}R_x, \qquad L_{x,y}=L_xL_yL_{yx}^{-1}, \qquad D_{x,y}=L_xR_yL_{xy}^{-1}.
$$
Clearly, when restricted to $A^\times$ these operators coincide with those defined in (\ref{ops})
and Lemma \ref{dxy}$(iv)$ for Moufang loops.

\begin{lem}\label{alt_op} Let $A$ be an alternative $R$-algebra. Then, for all $m,n,k\in A^\times$, we have
\begin{enumerate}
\item[$(i)$] $D_{m,n}=L_{m,n^{-1}}T_nL_{n,m}$;
\item[$(ii)$] $L_{n,m}D_{mn,km}=D_{n,k}L_{nk,m}D_{m\cdot nk,m}$;
\item[$(iii)$] $D_{k,m}L_{km,mn} = L_{k,n}L_{nk,m}D_{m\cdot nk,m}$;
\item[$(iv)$] $D_{m,n}D_{mn,km} + L_{m,k}L_{km,mn} =
                D_{m,nk}D_{m\cdot nk,m}+L_{m,m\cdot nk}$.
\end{enumerate}
\end{lem}
\begin{proof} $(i)$ This can be proved as Lemma \ref{dxy}$(v)$, where we only used the Moufang identities
which also hold in alternative algebras.

$(ii)$ By definition, we have
\begin{multline*}
L_{n,m}D_{mn,km}=L_nL_mL_{mn}^{-1}L_{mn}R_{km}L_{mn\cdot km}^{-1}=L_nL_mR_{km}L_{mn\cdot km}^{-1},\\
\shoveleft{D_{n,k}L_{nk,m}D_{m\cdot nk,m}=
L_nR_kL_{nk}^{-1}L_{nk}L_mL_{m\cdot nk}^{-1}L_{m\cdot nk}R_mL_{m\cdot nk \cdot m}^{-1} } \\
=L_nR_kL_mR_mL_{m\cdot nk \cdot m}^{-1}.
\end{multline*}
Since $A^\times$ is a Moufang loop, it remains to show that $L_mR_{km}=R_kL_mR_m$. This follows from
$ma\cdot km = (m\cdot ak)\cdot m$
for every $a\in A$, since the Moufang identities hold in $A$.

$(iii)$ We have
\begin{multline*}
D_{k,m}L_{km,mn}=L_kR_mL_{km}^{-1}L_{km}L_{mn}L_{mn\cdot km}^{-1}=L_kR_mL_{mn}L_{mn\cdot km}^{-1},\\
\shoveleft{L_{k,n}L_{nk,m}D_{m\cdot nk,m}=L_kL_nL_{nk}^{-1}L_{nk}L_mL_{m\cdot nk}^{-1}L_{m\cdot nk}
R_mL_{m\cdot nk \cdot m} }\\
=L_kL_nL_mR_mL_{m\cdot nk \cdot m}.
\end{multline*}
Again, it suffices to show that $R_mL_{mn}=L_nL_mR_m$. This follows from the Moufang identity
$mn\cdot am = (m\cdot na)\cdot m$.

$(iv)$ We have

\begin{multline*}
D_{m,n}D_{mn,km} + L_{m,k}L_{km,mn} = L_mR_nL_{mn}^{-1}L_{mn}R_{km}L_{mn\cdot km}^{-1}\\
\shoveright{ +L_mL_kL_{km}^{-1}L_{km}L_{mn}L_{mn\cdot km}^{-1}=
L_m(R_nR_{km}+L_kL_{mn})L_{mn\cdot km}^{-1}, }\\
\shoveleft{D_{m,nk}D_{m\cdot nk,m}+L_{m,m\cdot nk}=L_mR_{nk}L_{m\cdot nk}^{-1}
L_{m\cdot nk}R_mL_{m\cdot nk \cdot m}^{-1}}\\
+L_mL_{m\cdot nk}L_{m\cdot nk\cdot m}^{-1}=
L_m(R_{nk}R_m+L_{m\cdot nk})L_{m\cdot nk\cdot m}^{-1}.
\end{multline*}
The equality $R_nR_{km}+L_kL_{mn} = R_{nk}R_m+L_{m\cdot nk}$ follows from
$$
an\cdot km + mn\cdot ka = (a\cdot nk)\cdot m + (m\cdot nk)a
$$
for every $a\in A$, which in turn is the linearization of the Moufang identity $xn\cdot kx = (x\cdot nk)x$
with $x=m+a$.
\end{proof}

\begin{thm}\label{sd} Let $A$ be an alternative $R$-algebra and let $M$ be a subloop of $A^\times$.
Let $U$ a subgroup of the additive group of $A$ that is invariant under the operators
$T_m$ and $L_{n,m}$ for all $m,n\in M$.
Denote by $M\rd U$ the set of pairs $(m,u)$ for $m\in M$, $u\in U$. Then
with respect to the operation
\be\label{m_a}
(m,u)\cdot(n,w)=(mn,uD_{m,n}+wL_{n,m})
\ee
$M\rd U$ becomes a Moufang loop with identity $(1,0)$ and inversion
$$
(m,u)^{-1}=(m^{-1},-uT_{m}^{-1})
$$
\end{thm}
\begin{proof} First observe that $U$ is also invariant under $D_{m,n}$ due to Lemma \ref{alt_op}$(i)$; in
particular, $uD_{m,n}+wL_{n,m}$ lies in $U$.

We show that the Moufang identity $ab\cdot ca = (a\cdot bc)a$ holds for
arbitrary $a=(m,u)$, $b=(n,w)$, $c=(k,v)$ in $M\rd U$. We have
\begin{align*}
&ab\cdot ca = (mn,uD_{m,n}+wL_{n,m})\cdot (km,vD_{k,m}+uL_{m,k})\\
=\,&(mn\cdot km, (uD_{m,n}+wL_{n,m})D_{mn,km}+
(vD_{k,m}+uL_{m,k})L_{km,mn}).
\end{align*}
Also,
\begin{align*}
(&a\cdot bc)a = (a\cdot (nk,wD_{n,k}+vL_{k,n})) a =
(m\cdot nk, u D_{m,nk} +   (wD_{n,k}+vL_{k,n})L_{nk,m})a\\
=&\,((m\cdot nk)\cdot m, (u D_{m,nk} +   (wD_{n,k}+vL_{k,n})L_{nk,m})
D_{m\cdot nk,m} +u L_{m,m\cdot nk} ).
\end{align*}
The first components are equal, since $M$ is a Moufang loop. The second components are equal, since
the operators acting on $w,v,u$ coincide due to $(ii),(iii),(iv)$ of Lemma \ref{alt_op}, respectively.

The assertions about the identity and inversion are readily verified, once we note that
$D_{x,x^{-1}}=T_{x^{-1}}$, $D_{1,x}=T_{x}$, and $D_{x,1}=L_{x,x^{-1}}=L_{1,x}=L_{x,1}$
is the identity operator.
\end{proof}

{\em Remark 1.\/} Instead of considering $M$ as a subloop of $A^\times$ in Theorem \ref{sd},
we may clearly generalize this to an arbitrary loop homomorphism $M\to A^\times$.
In this situation we will say that the loop $M$ {\em acts} on the abelian subgroup $U$ of $A$ and
call $L=M\rd U$ an {\em outer semidirect product} of $M$ and $U$. We will identify $M$ and $U$
with their isomorphic images in $L$ consisting of elements $(m,0)$ and $(1,u)$, respectively.

{\em Remark 2.\/} Once we make this identification, the multiplication and inversion formulas for $L$
take the {\em inner} form
\be\label{mf_inn} \ba{r@{\,}l}
mu\cdot nw &= mn\cdot x,   \quad \text{with} \quad  x = uD_{m,n} + wL_{n,m},\\
(mu)^{-1} &= m^{-1}\cdot y,  \quad \text{with} \quad y = -uT_m^{-1},
\ea \ee
which means that the operators $T_m$, $D_{m,n}$, $L_{n,m}$ can be viewed as inner loop operators in $L$ without
reference to the alternative algebra $A$, thus agreeing with the formulas (\ref{xx_ab}).
Indeed, for example, $(1,u)$ is sent by the loop operator $L_{(m,0),(n,0)}$ to
\begin{multline*}
((nm)^{-1},0)((n,0)\cdot(m,0)(1,u))=((nm)^{-1},0)\cdot (n,0)(m,u)\\
=((nm)^{-1},0)\cdot (nm,uL_{m,n})=(1,uL_{m,n}),
\end{multline*}
and similarly for $D_{m,n}$ and $T_m$.

{\em Remark 3.\/} Let $M_0$ be the scalar multiples of $\I$ contained in $M$. It is easily checked that
$M_0$ is a central subloop of $M\rd U$ and we denote the factor loop $(M\rd U)/M_0$ by $\ov{M}\rd U$,
where $\ov{M}=M/M_0$.

{\em Remark 4.\/} If $U_0$ is a subgroup of $U$ invariant under
$T_m$, $L_{n,m}$
then $U_0$ is a normal subloop of $M\rd U$. In particular, this gives the Moufang loop
$(M\rd U)/U_0$ which we denote by $(M\rd \ov{U})$ with $\ov{U}=U/U_0$.
Combining with the previous remark, we also have a Moufang loop $\ov{M}\rd \ov{U}$.

\section{Semidirect products for simple Moufang loops \label{sml}}

We may use the previous results to construct outer semidirect product of
loops that arise from simple alternative algebras.

Let $A$ be a Cayley-Dickson algebra over a field $F$ equipped with a nondegenerate quadratic form $Q$ that
makes $A$ a composition algebra and an orthogonal space, see \cite[Chap. 2]{she}. Let $U=\I^\bot$, which is
a $7$-dimensional subspace of $A$. For all $m,n\in A^\times$
the operators $L_{m,n}$ and $T_m$ leave fixed the unit $\I$, and since they also preserve $Q$, they
leave invariant the subspace $U$. If $M\le A^\times$, by Theorem \ref{sd}, there exists the
Moufang loop $M\rd U$. Moreover, if $F$ has characteristic $2$,
we have $U_0=\la \I \ra\le U$.
Then $U/U_0$ is $6$-dimensional and we have the Moufang loop $M\rd \ov{U}$.

Denote by $\SL(A)$ the loop of units of $A$ of norm $1$ and by $\PSL(A)$ its quotient by the scalars.
The latter loop is not always simple, but it is such in the following two
important special cases \cite{pg}.
First, if $A$ is the classic division algebra of
octonions over $\RR$, then $\PSL(A)$ is simple and we have the Moufang loop $\PSL(A)\rd \RR^{\oplus 7}$.
Second, if $A$ is a split Cayley algebra over any field $F$. Then
$\PSL(A)$ is simple and we have the loops $\PSL(A)\rd F^{\oplus n}$, where $n=6$ or $7$ according as
$F$ has characteristic $2$ or otherwise.

In particular, the following semidirect products for finite simple Paige--Moufang loops $M(q)$ exist:
\begin{enumerate}
\item[$\bullet$] $M(q)\rd \FF_q^{\oplus 7}$, where $q$ is an odd prime power;
\item[$\bullet$] $M(q)\rd \FF_q^{\oplus 6}$, where $q$ is a power of $2$;
\item[$\bullet$] $M(2)\rd \FF_p^{\oplus 7}$, where $p$ is an odd prime.
\end{enumerate}
The extensions in the last exceptional case exist due to the
embedding of $M(2)$ as a maximal subloop into $M(p)$ for $p$
an odd prime, see \cite{gz_max}.

%
\section{Extensions of Moufang loops with abelian kernel \label{last}}

Let $M$ be a Moufang loop and $U$ an abelian group.
Suppose there is a short exact sequence of Moufang
loops
\be\label{es}
1\to U\to E \to M \to 1
\ee
We will identify $U$ with its image in $E$ and say that the extension (\ref{es}) of $M$
is {\em minimal}, if $U$ contains no subgroup that is a normal subloop of $E$.

The known nontrivial (i.\,e. nonassociative and not of the form $U\times M$)
minimal extensions of finite simple noncyclic Moufang loops are as follows:

\begin{enumerate}
\item[$\bullet$] Moufang semidirect products $G\rd V$ for a finite simple group $G$ listed in
Section \ref{mtg}.
\item[$\bullet$] Outer semidirect products $M\rd U$ for finite simple Moufang loops $M$ listed in
Section \ref{sml}.
\item[$\bullet$] Nonsplit central extensions $1\to \ZZ/2\ZZ  \to E \to M(q) \to 1$, where $q$ is an odd
prime power or $q=2$.
\end{enumerate}
The extensions in the last case are isomorphic to $\SL(A)$, where $A=\mathbb{O}(q)$ is the finite
Cayley algebra if $q$ is odd, and to the exceptional double cover of $M(2)$
of order $240$ if $q=2$, see \cite{gz_max}.
 The minimality of these extensions is quite apparent from their construction. We put forward

\begin{conj} Up to isomorphism, the only nontrivial minimal extensions for
finite simple noncyclic Moufang loops are those given in the list above.
\end{conj}

\end{document}